\documentclass[12pt, twoside, a4paper, reqno]{amsart}
\usepackage{amsmath, amssymb, latexsym, amsthm, amsfonts, mathrsfs, amscd}
\usepackage{url}
\theoremstyle{definition} % Bold heading, roman body
\DeclareMathAlphabet{\mathpzc}{OT1}{pzc}{m}{it}
\usepackage{enumerate}[1.)]
\renewcommand{\eqref}[1]{(\ref{#1})}   %for some reason eqref is not supported properly 
\usepackage{latexsym}
\usepackage{euscript}
\usepackage[margin=2.8 cm,bmargin=3 cm,tmargin=3 cm]{geometry}
\usepackage{epsfig}
\usepackage{hyperref}
\usepackage{graphics}
\usepackage{array}
\usepackage{enumerate}
\usepackage{url}
\usepackage{color}
\newtheorem{theorem}{Theorem}[section]
\newtheorem{lemma}[theorem]{Lemma}
\newtheorem{proposition}[theorem]{Proposition}
\newtheorem{corollary}[theorem]{Corollary}

\newtheorem{remark}{Remark}[section]

\newcommand{\starsum}{\sideset{}{^{\star}} \sum}
\newcommand{\starsump}{\sideset{}{^{\star}} \sum_{h \modp}}
\renewcommand{\l}{\left}
\renewcommand{\r}{\right}
\newcommand{\noi}{\noindent}
\newcommand{\moda}{(\textnormal{mod }a)}

\newcommand{\modx}{(\textnormal{mod } a_1)}
\newcommand{\modq}{(\textnormal{mod }q)}
\newcommand{\modp}{(\textnormal{mod }p)}

\newcommand{\diff}{\textnormal{d}}

\newcommand{\R}{{\mathbb R}}
\newcommand{\Z}{{\mathbb Z}}
\newcommand{\C}{{\mathbb C}}

\newcommand{\h}{{\mathbb H}}
\newcommand{\F}{{\mathbb F}}

\newcommand{\ov}{\overline}
\newcommand{\be}{\begin{equation}}
	\newcommand{\ee}{\end{equation}}
\newcommand{\ba}{\begin{equation}\begin{aligned}}
		\newcommand{\ea}{\end{aligned}\end{equation}}
\numberwithin{equation}{section}

\newcommand{\abcdsumV}{\mathop{\sum_a\sum_b\sum_c\sum_d}_{\substack{ ad-bc=r }}}

 2

\newcommand{\ve}{\varepsilon}

\begin{document}
\title[Integer points]  {Distribution of integer points on determinant surfaces and a 
		$\text{mod-}p$ analogue}
	\author{Satadal Ganguly, Rachita Guria}
	\address{Theoretical Statistics and Mathematics Unit, Indian Statistical Institute, 203 Barrackpore
		Trunk Road, Kolkata-700108, India.}
	\email{sgisical@gmail.com}
	\email{guriarachita2011@gmail.com}
	\subjclass[2020] {Primary 11E20, 11E25, 11F30, 11F72, 11N45, ; Secondary 11L05, 11L07}
	\keywords{ Representations by quadratic forms, Fourier coefficients of cusp forms, Kloosterman sums, Kuznetsov formula}
	%\author{}
	%\address{Indian Statistical Institute, Kolkata}
	%\email{}
	
	%\blfootnote{2010 {\it Mathematics subject classification}: Primary 11M06; Secondary 11N36}

	\begin{abstract}
		We establish an asymptotic formula for counting integer solutions with smooth weights to an equation of the form $xy-zw=r$, where $r$ is a non-zero integer, with an explicit main term and a strong bound on the error term in terms of the size of the variables $x, y, z, w$ as well as of $r$. We also establish an asymptotic formula for counting integer solutions with smooth weights to the congruence $xy-zw \equiv 1 (\text{mod }p)$, where $p$ is a large prime, with a strong bound on the error term. 
	\end{abstract}
		\maketitle		
	\section{Introduction}
	
	From the time of Fermat or perhaps even earlier, the question of representing integers by  quadratic forms has fascinated  mathematicians. Let $Q$ be a fixed non-singular quadratic form in $n$ variables over $\Z$ and let $r$ be a given integer. The question is to  determine if there are solutions $\mathbf{m}=(m_1, m_2, \cdots, m_n)\in {\Z}^n$ to the Diophantine equation
	\be\label{equation}
	Q(\mathbf{m})=r,
	\ee
	and if there are, then to count the number of such solutions. 	There are algebraic, analytic and geometric approaches towards studying this question and its various generalizations and these continue to be active fields of research. There are two major analytic techniques that have been utilised to tackle this question. One goes via the theory of modular forms pioneered by Jacobi through his use of theta series and  developed mostly by Siegel and his followers (see \cite{Kitaoka}). The other is via the circle method, invented and refined by Hardy, Ramanujan, Littlewood, Kloosterman and 
	many others. See \cite {HB} for a modern account of this method applied to this problem.\\
	
	\noi
	An interesting issue that arises when studying representation of integers by quadratic forms is that  the local-global or the Hasse principle  does not hold in this set-up in general. There are forms that represent
	an integer locally, i.e., over $\R$ and over $\Z_p$, but do not represent it globally, i.e., over $\Z$. See, e.g.,  \cite[ex. 23, pg. 168]{Cassels}. Colliot-Thelene and  Xu \cite{CS} have given 
	a conceptual explanation for the occurrences of such failures of the Hasse principle through their theory of integral Brauer-Manin obstruction.
	There is a result of Siegel, however, that says that the integral Hasse principle does hold for representation of integers by indefinite quadratic forms as long as  $n\geq 4$. For a proof, see \cite{Siegel-indef} or \cite[Thm. 1.5, Chap. 9]{Cassels}. See also Remark \ref{general} below for a different approach.  \\
	
	\noi
	There is an important distinction between definite and indefinite forms in regard to this question. Definite forms are easier to study in this respect as the sizes of the variables are automatically restricted and we can  obtain a precise count for the original problem with a good error term if $n\geq 4$. This is explained in \cite{HB}; see, in particular,  \cite[Cor. 1]{HB}. One may also invoke the theory of Modular Forms \cite[Chap. 11]{Iwa-classical} to obtain an asymptotic formula with a strong error term in this case.\\
	If the form is indefinite then there can be an infinite number of integer solutions to \eqref{equation} and hence to make sense of the counting problem we impose a condition that restricts the heights of the integer points. By \textit{height} we mean the maximum of the absolute values of the coordinates. So we introduce a positive parameter $X$ increasing to infinity and then ask for the asymptotic for the number of integer points that satisfy \eqref{equation} and have heights of size about $X$. To be precise, one may impose the condition that the heights should lie between $X$ and $2X$. Such ``sharp cuts" conditions lead to obvious analytic difficulties and to circumvent that we often count with smooth weights instead, i.e., we replace the indicator function 
	of $(X, 2X)$ by $w(x/X)$,  where $w$ is a $C^{\infty}$ function with compact support in $[1,2]$. Thus we change the original problem and seek an  asymptotic formula for the ``smoothed" sum
	\[
	N(Q, \mathbf{w}, X, r)=  \sum_{Q(\mathbf{m})=r} \mathbf{w}(\mathbf{m}),
	\]
	where  $\mathbf{w}(\mathbf{x})=\prod_i w_i(x_i/X)$ and each $w_i$ is a $C^{\infty}$ function with compact support in $[1, 2]$.  Such formulae are quite useful in studying sums of the form 
	$\underset{\mathbf{m}}{\sum} \mathbf{w}\l(Q(\mathbf{m}\r).$\\
	Our goal in this article is to obtain an asymptotic formula for  $N(Q, \mathbf{w}, X, r)$ with an explicit main term and a particularly strong error term for the specific indefinite form $Q(x_1, x_2, x_3, x_4)=x_1 x_2-x_3 x_4$. Apart from the intrinsic interest, the  problem of counting solutions to the determinant equation comes up quite naturally and  frequently in Analytic Number Theory, e.g., in studying moments of $L$-functions (see \cite{BC}, \cite{DFI2}), and this problem is related to other well-studied problems such that of estimating shifted convolution sums of the divisor function.  These connections are explained in the introduction to our companion paper \cite{GG1}. 
	Here we consider  $X$ and $r$ as  independent variables. This is unlike the situation in \cite{HB} in which only  two cases are considered: either $r=0$ or  $r=X^2$. However, allowing $r$ to have a greater leeway is not without interest, as explained in Remark \ref{general}.  For simplicity, we use the same weight function for each of the variables. We now state our main result. 
	
	\begin{theorem} \label{main2}
		Suppose $V$ be a smooth function on $\R$ with compact support inside $[1, 2]$ and let
		\[
		S_{V}(X,r)  = \abcdsumV V\l(\frac{a}{X}\r)V\l(\frac{b}{X}\r)V\l(\frac{c}{X}\r)V\l(\frac{d}{X}\r),
		\]
		where $r$ is a nonzero integer. 
		Then we have,
		\ba
		S_{V}(X,r) = \mathcal{M} _V(X,r) + O_{\ve}\l( r^{\theta}X^{1+\ve}\r) ,
		\ea
		where 
		\be 
		\mathcal{M}_V(X,r) = \sum_{l|r}\sum_{k>0}\frac{\mu(k)}{k}\iiint \frac{1}{z} V\l(\frac{x}{X}\r)V\l(\frac{lky}{\ X}\r)V\l(\frac{lkz}{X}\r)V\l(\frac{r+lkxy}{ zlkX}\r) \diff x \diff y \diff z,
		\ee
		and $\theta$ is the exponent occurring in the Ramanujan-Petersson conjecture (see \eqref{ramanujan}). Conjecturally $\theta = 0$ and the current record due to Kim and Sarnak (see \cite{KS}) is $\theta\leq 7/64$.
	\end{theorem}
	Note that the support of $V$ forces that  we must have $|r|<3 X^2$ for at least one solution to exist;  and it also follows that $ 	\mathcal{M}_V(X,r) \asymp X^2$ if  $-3X^2 <r<X^2$ if we assume that $V(x)>0$ for $1<x<2$. To give a more precise main term, we need to be more specific about the weight function $V$. However, if the order of $|r|$  is a little lower compared to that of $X$ then we obtain a very precise main term as the following corollary shows. 
	\begin{corollary}\label{maincorollary} Under the same assumptions as in Theorem \ref{main2}, we have
		\ba 
		S_{V}(X,r) = 	K(V, r) X^2
		+ O\l( \sigma_1(|r|) \r)+ O\l(|r|^{\theta} X^{1+\ve}\r),
		\ea    
		where
		\be
		K(V, r) = \frac{1}{\zeta(2)}\frac{\sigma_1(|r|)}{|r|} \iiint\frac{1}{t} V\l(u\r)V\l(v\r)V\l(t\r)V\l(\frac{uv}{ t}\r) \diff u\, \diff v \,\diff t,
		\ee
where $\sigma_1(n) = \sum_{d|n} d$.
	\end{corollary}
	In particular,  $S_{V}(X,r) \sim 	K(V, r) X^2$ as long as $r =o(X^2)$. \\
	\noi
	For a general quaternary form, the best known bound on the error term is $O(X^{3/2+\ve})$. We could not find this stated anywhere in the literature except in the special cases $r=0$ and $r=X^2$ which are the only cases treated in \cite{HB}. However, for any choice of $r$ with $|r|\leq X^2$, an asymptotic formula with this error term can be obtained  by  the method in \cite{HB}. 
	More generally, an analysis of the proofs in \cite{HB} allows us to make the following remark. 
	\begin{remark}\label{general}
		Theorems 5, 6, 7 	in \cite{HB}  continue to hold even  if $0\neq |r|\leq X^2$, with $\delta=0$ in Theorem 5. 
		Since there always exists a real solution to $Q(\mathbf{x})=r$ for any integer $r$ for an indefinite form $Q$. we obtain a different proof of  Siegel's theorem\textemdash in fact, a quantitative version of it\textemdash about representing integers by indefinite quadratic forms alluded to earlier. 
	\end{remark}
	\begin{remark}
		A crucial ingredient in the proofs of these theorems is exploitation of the cancellations in sums over $q$, the modulus (see \cite[Thm. 2]{HB} for the notation). This technique has been  called ``the double Kloosterman refinement" in the literature (see \cite{Hooley}, \cite{HB}). Exploiting the $q$-averaging technique further, it is possible to improve upon the error terms appearing  in \cite{HB} as  Getz has demonstrated (see \cite[Thm. 1.1]{Getz}); but this requires the condition that $r=0$. 
		The difficulty with the case $r\neq 0$ is that $S_q(\mathbf{c})$ is not multiplicative as a function of $q$, except in the case $\mathbf{c}=\mathbf{0}$. Therefore,  we can take advantage of the cancellations in the sum over $q$ \textit{only in the zero frequency}.   This is sufficient if $n\geq 4$ as  the Weil bound applied to the non-zero frequencies produces a sufficiently small error term, compared to the main term. 
		For a quadratic form in three variables, we cannot obtain an asymptotic formula for $N(Q, \mathbf{w}, X, r)$ in this way if $r\neq 0$. In this case, the bound on the error term that the Weil bound yields is $O(X^{1+\ve})$,  whereas the main term grows like $X$.  In particular, we do not have the analogue of \cite[Thm 8]{HB} for $r\neq 0$.
	\end{remark}
	\begin{remark}
		In our earlier work \cite{GG1}, we  considered a similar counting problem but we viewed it as a lattice point counting problem and, therefore, we did not attach any  smooth weights. Precisely, we obtained an asymptotic formula for the
		exact  number of integer points $(a, b, c, d)$ inside an expanding box $[-X, X]^4$ such that $ ad- bc=r$, where $r$ is a fixed non-zero integer satisfying $|r|\leq X^{1/3}$ with a precise main term and  an error term of the order 
		$O(|r|^{\theta}X^{3/2+\ve})$ (see \cite[Thm. 1.1]{GG1}).  Needless to say, the proofs of the two results are substantially different from one another  though the basic idea of both the proofs are the same. Obtaining an error term of this strength for a p-cut sum is a challenging problem and requires  an extremely delicate analysis and several technical innovations. Furthermore, compared to the proof of Theorem \ref{main2} above, many other results from the Spectral Theory of Automorphic Forms are invoked in the proof of \cite[Thm. 1.1]{GG1}.  One can probably enlarge the range of $r$ keeping the same error term with more work and it can certainly be done  at the cost of a worse error term  (see \cite[Rmk. 7.6]{GG1} and also \cite[Thm. 1.1]{Afifur}).
	\end{remark}
	
	\begin{remark}
		To prove Theorem \ref{main2}, instead  of applying the circle method as has been done in \cite{HB}, we apply 
		Poisson summation directly after interpreting the equation as a congruence modulo one of the variables and this 
		transforms the original sum to a sum of Kloosterman sums due to the particular shape of the equation. As one might 
		expect, we next invoke the spectral theory of automorphic forms to estimate the resulting sum via the Kuznestov 
		trace formula. The main difficulty is to obtain a strong bound in terms of the size of $r$, and this necessitates a 
		careful analysis of the Bessel transforms that appear in the Kuznetsov formula.  Our method can also handle 
		equations  of the form $\alpha x_1 x_2-\beta x_3 x_4=r$, where $\alpha, \beta$ are integers with $\alpha\beta>0$ 
		and $(\alpha, \beta)|r$. In fact, this method should also extend to ternary forms of the type 
		$f(x_1, x_2, x_3)=\alpha x_1 x_2 - \beta{x_3}^2 $ with $\alpha \beta> 0$, and in this case, instead of  
		Kloosterman sums we will get Sali\'e sums and Proskurin's formula \cite{Proskurin} should play the same role 
		as the Kuznetsov formula.  
	\end{remark}

	\subsection{The mod-$p$ analogue}
	Our second question concerns the existence of  ``many" matrices of small heights that reduce modulo $p$ to an unimodular matrix. Here $p$ is a large prime and  by the \textit{height} of a matrix, we mean the maximum of the absolute values of its entries. The most natural way to approach this question is to obtain an asymptotic count for the number of $n\times n$ integer matrices of height $X$, a large positive real number,  that reduce modulo  $p$ to an element of $SL_n ({\F}_p)$.  One would like to  obtain an asymptotic formula for the smallest possible $X$ in terms of $p$. Ahmadi and Shparlinski studied this and several other similar interesting questions in \cite{AS}.  For  $n=2$,  this amounts to counting integer vectors $(x, y, z, w)$
	inside the box $[-X, X]^4$  satisfying the congruence $xy-zw\equiv1\modp$.   Certainly, this set is non-empty as there are solutions such as $(1, 1, 0, 0)$ or $(-1, -1, 0, 0)$  but one would like to obtain an asymptotic for this quantity for the least possible $X$ in terms of $p$. In \cite{AS}, the authors were able to show that it is enough to take $X\gg p^{1/2+\ve}$. 
	Specifically, they showed (see \cite[Thm. 12]{AS})
	\be\label{AS}
	\mathop{\sum\sum\sum\sum}_{\substack{|x|, |y|, |z|, |w|\leq X\\xy-zw\equiv 1 \modp}}1  = \frac{(2X+1)^4}{p}+O\l(X^2 p^{\ve}\r),
	\ee
	for any $\ve>0$. Thus we obtain the expected asymptotic $\frac{(2X)^4}{p}$ as soon as 
	$X\gg p^{1/2+\ve}$. As has been our custom throughout this article, $\ve$ denotes  a positive real number that can be taken to be arbitrarily small. In particular, it follows that 
	there are many integer matrices  of height $O(p^{1/2+\ve})$ that reduce modulo $p$ to elements of $SL_2 ({\F}_p)$. Our goal here is to show that this holds even if we replace the factor $p^{1/2+\ve}$ by  $h(p)$, where $h$ is any function from the set of positive real numbers to itself such that $\frac{h(x)}{\sqrt{x}}\to \infty$ as $x\to \infty$, no matter how slowly. We achieve this by replacing the original sum they have considered by the corresponding ``smoothed" sum which is much easier to estimate by Fourier analysis. Here is our result. 
	\begin{theorem}\label{modpthm}
		Let $V$ be a smooth function with compact support inside $[1,2]$. Suppose $p$ is a large prime and $X$ is a positive real number satisfying $p^{1/100}< X< \frac{p}{2}$.  Then we have, 
		\be\label{modpeqn}
		\mathop{\sum\sum\sum\sum}_{\substack{ ad-  bc\equiv 1 \modp}}
		V\l(\frac{a}{X}\r)V\l(\frac{b}{X}\r)V\l(\frac{c}{X}\r)V\l(\frac{d}{X}\r)
		= \frac{X^4}{p} \l(\int V(t)\diff t\r)^4+O(X^2). \nonumber
		\ee
	\end{theorem}
	
	\begin{remark}
		Note that we need to take $X$ to be any order larger than $\sqrt{p}$ in order to have an asymptotic formula. In fact, we cannot make  $X$  much smaller than this. If $X < \sqrt{p}$, then the congruence becomes an equality, and then we are in the regime of Theorem \ref{main2}. If $X< C\sqrt{p}$ for some constant $C>1$, then the congruence amounts to a bounded number of equalities, and again we can apply Theorem \ref{main2} to each of them. 
	\end{remark}
	
	\begin{remark}
		The bounds on $X$ that we have imposed is to make the proof technically simpler and can be removed. The  range of interest is, of course,  when $X$ is near $\sqrt{p}$. 
	\end{remark}
	
	\subsection{Notations and conventions}
	Our notations are quite standard and are similar to our companion paper \cite{GG1} where a detailed description of the notations are given. The notation $\widehat{f}$ denotes the Fourier transform of a  function $f$. The letter $\ve$ denotes a positive real 
	number which can be taken to be smaller than any given positive real number and in different occurrences it may assume different values. In particular, we will always assume that $0<\ve<1/1000$. The dependency of $\ve$ on the implied constant in an expression of the form $f(X)=O_{\ve}(X^{\ve})$ has been  suppressed in order to avoid cluttering the expressions. The symbol $\sim$ denotes asymptotic as usual; but in the appropriate context it is also used to denote the range of a variable running over a dyadic segment; i.e., $n\sim N$ means $N\leq n<2N$. For $\alpha \in \C$,  $\sigma_{\alpha}(n)$ denotes the sum \(\sum_{d|n} d^{\alpha}\).

	\subsection*{Acknowledgements}
	The authors thank  Tim Browning, \'Etienne Fouvry,  Kummari Mallesham, Ritabrata Munshi,  and Olivier Ramar\'e for interesting and useful conversations, suggestions and encouragement. They are specially thankful to Igor Shparlinski for several interesting comments on this work and for pointing out a few errors in an older version of this paper.  The authors also thank the Indian Statistical Institute, Kolkata and  the Max Plank Institute for Mathematics, Bonn where parts of this work were carried out, for their excellent working conditions. Finally, the authors profusely thank the anonymous referee for a careful reading of an earlier draft and for numerous corrections and suggestions that improved the quality and the readability of this article. A weaker version of Theorem \ref{main2}  appeared in the IISER Kolkata PhD thesis of the second author in 2023. 
	
	\section{Preliminaries}
	In this section, we recall some standard results from the literature which will be required in the proofs of our results and some applications of the formula.
	\subsection{Analytic Results}
	\begin{lemma} \label{poissonlemma}
		(\textbf{Poisson summation
			formulae})\\
		Suppose that both $f$, $\widehat{f}$ are in $L^1(\R)$ and have bounded variation. Let $\alpha, q$ be integers with $q>1$. Then, we have
		\be \label{poissonmain}
		\sum_{n \in \Z}f(n) = \sum_{n \in \Z}\widehat{f}(n) ,
		\ee
		\be\label{10}  
		\sum_{\substack{n \equiv \alpha \modq \\ n \in \Z}}f\left(n\right)= \frac{1}{q}\sum_{n \in \Z}e \left(\frac{\alpha n}{q}\right)\widehat{f}\left(\frac{n}{q}\right), 
		\ee 
	\end{lemma}
	\begin{proof}
		See, e.g., \cite[Theorem 4.4 and Exercise 4]{IK}.
	\end{proof}
	\begin{lemma} \label{6}
		Let $a, q$ be integers with $q>1$. Suppose $g$ and its Fourier transform $\widehat{g}$ are both in $L^1 (\R)$ and both
		have bounded variation. Then we have,
		\be 
		\sum_{(n,q)=1}g(n)e\left(\frac{a\overline{n}}{q}\right) = \frac{1}{q}\sum_n \widehat{g}(n/q)S(n, a,q),
		\nonumber 
		\ee 
		where 
		\be
		S(m,n,q) = \starsum_{\beta \modq} e\l( \frac{m\beta +n\overline{\beta}}{q}\r). \nonumber
		\ee
		is the Kloosterman sum (see \cite{IK}).
	\end{lemma}
	\begin{proof}
		This follows from subdividing the sum over $n$ into residue classes coprime to $q$ and then applying \eqref{10}. 
	\end{proof}
	\begin{lemma}\label{stir}
		({\textbf{The Stirling formula}})\\
		The Stirling asymptotic formula
		\be \label{stirone}
		\Gamma(s)= \l(\frac{2\pi}{s}\r)^{\frac{1}{2}}\l(\frac{s}{e}\r)^s\l(1+O\l(\frac{1}{|s|}\r)\r)
		\ee
		is valid in the angle $|\arg s|\leq \pi - \ve$, with the implied constant depending on $\ve$. Hence for $s=\sigma+it$, $t\neq 0$, $\sigma$ fixed
		\be \label{stirtwo}
		\Gamma(\sigma +it)= \sqrt{2\pi}(it)^{\sigma -\frac{1}{2}}e^{-\frac{\pi}{2}|t|}\l(\frac{|t|}{e}\r)^{it}\l(1+O\l(\frac{1}{|s|}\r)\r).
		\ee
	\end{lemma}
	
	\begin{proof}
		See \cite[Eq. (5.113)]{IK}.
	\end{proof}

	\subsection{Results from the theory of automorphic forms}	
	\subsubsection{Holomorphic forms}\label{holo}
	We recall some standard results and also fix the notations and normalisations for later use. For even integers $k>0$, let $ \{h_{j,k}(z): 1\leq j \leq$ dim $S_k(\Gamma_0(1))\}$ be an orthonormal Hecke basis for the space of holomorphic cusp forms $S_k(\Gamma_0(1))$ with each $h_{j,k}(z)$ having the Fourier expansion.
	\[
	h_{j,k}(z) = \sum_{n=1}^{\infty} \psi_{j, k}(n)n^{\frac{k-1}{2}}e(nz).
	\]
	Setting
	$
	\rho_h(n) = \left(\frac{\Gamma(k-1)}{(4\pi)^{k-1}}\right)^{\frac{1}{2}}\psi_h (n),
	$
we can write, for a Hecke eigenform $h$,
	\[
	\rho_h(n) =\rho_h(1) \lambda_h(n),
	\]
	where $\lambda_h(n)$ is the eigenvalue of the $n$-th Hecke operator applied to the form $h$ and the Hecke eigenvalues satisfy the following bound proved by Deligne:
	\begin{equation}\label{Deligne}
		|\lambda_h(n)| \leq \tau(n),
	\end{equation}
	where $\tau(n)$ is the total number of divisors of $n$. 
	Let us now recall Theorem 3.6 in \cite{Iwa-classical}.
	\begin{theorem}(\textbf{The Petersson trace formula})
		For $m,n >0$, we have
		\begin{align}\label{Petersson}
			\sum_{1\leq j \leq \textnormal{dim} S_k(\Gamma_0(1))}\overline{\rho_{jk}}(m)\rho_{jk}(n)
			= \delta(m,n) + 2\pi i^k \sum_{c>0}^{\infty}\frac{S(m,n;c)}{c} J_{k-1}\left( \frac{4\pi \sqrt{mn}}{c}\right).
		\end{align}
	\end{theorem}
	\subsubsection{Maass forms}
	Now we discuss the Maass forms. Let us fix an orthonormal basis of Maass cusp forms $\{u_j(z)\}_1^{\infty}$ for the full modular group $  \text{SL}(2, \Z)$, consisting of  common eigenfunctions of all the Hecke operators $T_n, n\geq 1$. Thus every $u_j(z)$ is an eigenfunction of the Laplace operator with the associated eigenvalue $s_j(1-s_j)$, $s_j \in \C$. We write $s_j$ as $s_j=\frac{1}{2}+i\kappa_j$ with $\kappa_j \in \C$. Since we are working with the full modular group, the Selberg eigenvalue conjecture is known, i.e., $\kappa_j \in \R$ and in fact $|\kappa_j| > 3.815$ (see, e.g., \cite[Lemma 1.4]{motospectral}).
	Every $u_j(z)$ has the following Fourier expansion at the cusp at $\infty$:
	\be 
	u_j(z) = \sqrt{y} \sum_{n \neq 0} \rho_j(n)K_{i\kappa_j}(2\pi |n|y)e(nx), \nonumber
	\ee 
	where $z = x+iy,$ $x,y \in \R$. Note that $K_{i\kappa}(y) = K_{-i\kappa}(y)$. Hence, without any loss of generality, we may assume that $\kappa_j >0$ as we shall do below. 
	Suppose $\lambda_j(n)$ is the eigenvalue of $u_j(z)$ for the Hecke operator $T_n$. Then the following Hecke relation
	\ba \label{heckerelation}
	\lambda _j (n)\lambda _j (m) = \sum_{d |(m,n)} \lambda _j \l(\frac{mn}{d^2}\r)
	\ea
	holds. We also have the relations
	\ba \label{realtion}
	\rho_j(\pm n) = \rho_j(\pm 1)\lambda_j(n), \quad \rho_j(-1) = \ov{\rho_j(1)}.
	\ea
	\textbf{The Ramanujan-Petersson conjecture:} This is the assertion that the Hecke eigenvalues at primes $p$ satisfy the bound 
	\ba \label{ramanujan}
	H(\theta): |\lambda_j(p)| \leq 2 p^{\theta},
	\ea
	for any  $\theta\geq 0$. 
	\begin{remark} \label{kim}
		By the work of Kim and Sarnak \cite[Appendix 2]{KS}, $H(\frac{7}{64})$ is known to hold.
	\end{remark}
	\begin{lemma}\label{hl} We have,
		\ba
		\frac{|\rho_j(1)|^2}{\cosh(\pi \kappa_j)} \ll_{\ve} \kappa_j ^{\ve}.\nonumber
		\ea
	\end{lemma}
	\begin{proof}
		See \cite[Corollary 0.3]{HL}.
	\end{proof}
	
	\begin{lemma}\label{kuz}(\textbf{The Kuznetsov formula}) 
		(a) Let $f$ be of class $C^2$ with compact support in $(0,\infty)$. Then for integers $m$ and $n$ with $mn<0$, we have 
		\ba \sum_{c>0} \frac{S(m,n,c)}{c} f\l(\frac{4\pi \sqrt{|mn|}}{c}\r)&= \sum\limits_{j = 1}^{\infty} \rho_ j(n) \ov{ \rho_ j(m) } \check{f}(\kappa_ j) \\
		&+ \frac{1}{\pi}\int\limits_{-\infty}^{\infty} (nm)^{-i\eta}\sigma_{2i\eta}(n)\ov{\sigma_{2i\eta}(m) } \frac{\cosh(\pi \eta)\check{f}(\eta)}{\l|\zeta (1+2i\eta)\r|^2}\diff \eta ,\nonumber
		\ea 
		where $\check{f}(\eta)$ is the Bessel transform 
		\be 
		\check{f}(\eta)= \frac{4}{\pi}\int\limits_ 0 ^{\infty} K_{2i\eta}(t) f(t) \frac{\diff t}{t}.  \nonumber
		\ee
		(b) Under the same conditions, if $m$ and $n$ are integers such that $mn>0$, then we have 
		\ba \sum_ c \frac{S(m,n,c)}{c} f\l(\frac{4\pi \sqrt{mn}}{c}\r)&= \sum\limits_{j = 1}^{\infty} \rho_ j(n)  \ov{ \rho_ j(m) } \ddot{f}(\kappa_ j) \\
		&+ \frac{1}{\pi}\int\limits_{-\infty}^{\infty} (nm)^{-i\eta}\sigma_{2i\eta}(n)\ov{\sigma_{2i\eta}(m) }\frac{\cosh(\pi \eta)\ddot{f}(\eta)}{\l|\zeta (1+2i\eta)\r|^2}\diff \eta	\\
		&+\sum_{0<k\equiv 0(\textnormal{mod }2)} \widetilde{f}(k)\sum_{1\leq j \leq \textnormal{dim} S_k(\Gamma)} \overline{\psi_{jk}}(|m|)\psi_{jk}(|n|),\nonumber
		\ea	
		where 
		\[
		\ddot{f}(\eta)= \frac{\pi i}{\sinh 2\pi \eta }\int\limits_{0}^{\infty}(J_{2i\eta}(x) - J_{-2i \eta}(x))f(x)\frac{\diff x}{x},
		\]
		and 
		\[
		\widetilde{f}(k)=\frac{4(k-1)!}{(4\pi i)^k}\int_{0}^{\infty} J_{k-1}(x)\frac{f(x)}{x}\diff x.
		\]
	\end{lemma}
	\begin{proof}
		See \cite{DI1} or \cite[Chap. 16]{IK}.
	\end{proof}
	\begin{lemma}\label{weyl}(\textbf{The Weyl law})
		We have,
		\be 
		N_{SL(2,\Z)}(K) = \l|\{j: |\kappa _j| \leq K \}\r| =  \frac{\textnormal{Vol}(SL(2,\Z) \backslash \h)}{4 \pi} K^2 + O(K\log K). \nonumber
		\ee
	\end{lemma}
	\begin{proof}
		See \cite[Eq. (11.5)]{IW2}.
	\end{proof}
	\begin{lemma} \label{hec}
		For any $\ve > 0$ and all $N \geq 1$, we have
		\ba 
		\sum_{n \leq N} |\lambda _j(n)|^2 \ll_{\ve} (\kappa_j)^{\ve}N. \nonumber
		\ea
	\end{lemma}
	\begin{proof}
		See \cite[Lemma 1]{IW1}.
	\end{proof}
	\section{Beginning of the proof}
	Our basic idea is to apply Fourier analysis to study the sum $S_{V}(X,r)$. For this, we replace $d$ by the expression $(r+bc)/a$ and  rewrite  the condition $a |(r+bc)$ as the congruence condition $bc \equiv -r \moda$. Thus, 
	\ba\label{eqcgr}
	S_{V}(X,r) &= \abcdsumV V\l(\frac{a}{X}\r)V\l(\frac{b}{X}\r)V\l(\frac{c}{X}\r)V\l(\frac{d}{X}\r) \\
	&=\mathop{\sum\sum}_{(a,c)|r} V\l(\frac{a}{X}\r)V\l(\frac{c}{X}\r) \sum_{\substack{b \in \Z \\ a| (r+bc)}}V\l(\frac{b}{X}\r)V\l(\frac{r+bc}{aX}\r)\\
	&= \sum_{l|r}\mathop{\sum\sum}_{(a_1,c_1) = 1}V\l(\frac{la_1}{X}\r)V\l(\frac{lc_1}{X}\r) \sum_{\substack{b \in \Z \\ b\equiv -r_1\overline{c_1}\modx}}V\l(\frac{b}{X}\r)V\l(\frac{r_1+bc_1}{a_1X}\r),
	\ea where $l = (a,c)$, $r_1=r/l$, $c_1=c/l$, $a_1=a/l$. Here, $\ov{c_1}$ denotes the multiplicative inverse of $c_1$ modulo $a_1$.\\
	By Poisson summation, i.e., \eqref{10} applied to the sum over $b$, we have, 
	\[
	\sum_{\substack{b \in \Z \\ b\equiv -r_1\overline{c_1}\modx}}V\l(\frac{b}{X}\r)V\l(\frac{r_1+bc_1}{a_1X}\r) =\frac 1 {a_1}\sum_{n\in \Z}e\left(\frac{-nr_1\overline{c_1}}{a_1}\right)\widehat{V_{a_1,c_1}}\left(\frac{n}{a_1}\right),
	\]
	where 
	\[
	V_{a,c}(x)=V\l(\frac{x}{X}\r)V\l(\frac{r_1+xc}{aX}\r),
	\]
	Next, we separate the term $\widehat{V_{a_1,c_1}}\left(0\right)$  from which we shall isolate the main term later. Hence, we write
	\[
	S_{V}(X,r)=M_{V}(X,r)+E_{V}(X,r),
	\]
	where
	\[
	M_{V}(X,r)= \sum_{l|r}\mathop{\sum\sum}_{(a_1,c_1) = 1}\frac{1}{a_1} V\l(\frac{la_1}{X}\r)V\l(\frac{lc_1}{X}\r)\widehat{V_{a_1,c_1}}\left(0\right),
	\]
	and
	\[
	E_{V}(X,r)= \sum_{l|r}\mathop{\sum\sum}_{(a_1,c_1) = 1}\frac{1}{a_1}V\l(\frac{la_1}{X}\r)V\l(\frac{lc_1}{X}\r) \sum_{\substack{n\in \Z\\ n\neq 0}}e\left(\frac{-nr_1\overline{c_1}}{a_1}\right)\widehat{V_{a_1,c_1}}\left(\frac{n}{a_1}\right).
	\]
	
	\section{The Main Term}
	We now isolate a main term from $M_{V}(X,r) $ as shown below.
	\begin{proposition} \label{MTsmooth}
		We have,
		\ba
		M_{V}(X,r) & =\sum_{l|r}\sum_{k>0}\frac{\mu(k)}{k}\iiint \frac{1}{z} V\l(\frac{x}{X}\r)V\l(\frac{lky}{X}\r)V\l(\frac{lkz}{X}\r)V\l(\frac{r+lkxy}{zlkX}\r) \diff x \diff y \diff z \\
		& + O\l(\tau (r)X^{1+\ve}\r) .
		\ea
	\end{proposition}
	\begin{proof}
		
		Opening up $\widehat{V_{a_1,c_1}}\left(\frac{n}{a_1}\right)$  and accumulating the $c_1$-terms and removing the coprimality condition by M\"{o}bius inversion, we have
		\ba 
		M_{V}(X,r)& = \sum_{l|r}\int V\l(\frac{x}{X}\r) \sum_{a_1}\frac{1}{a_1}V\l(\frac{la_1}{X}\r)\sum_{\substack{c_1 \\ (a_1,c_1)=1}} V\l(\frac{lc_1}{X}\r)V\l(\frac{r_1+xc_1}{a_1X}\r)\diff x\\
		& =\sum_{l|r}\sum_{k>0}\frac{\mu(k)}{k}\int V\l(\frac{x}{X}\r) \sum_{a_1}\frac{1}{a_1}V\l(\frac{lka_1}{X}\r)\sum_{c_2 }W(c_2, a_1) \diff x, 
		\ea 
		where we have introduced the notation
		\[
		W(\alpha, \beta)=W_{l, k, x, X, r_1}(\alpha, \beta)=V\l(\frac{lk\alpha}{X}\r)V\l(\frac{r_1+xk\alpha}{k\beta X}\r).
		\]
		Applying the Poisson summation formula (see \eqref{poissonmain}) to the sum over $c_2$, we get
		\ba
		M_{V}(X,r)=\sum_{l|r}\sum_{k>0}\frac{\mu(k)}{k}\int V\l(\frac{x}{X}\r) \sum_{a_1}\frac{1}{a_1}V\l(\frac{lka_1}{X}\r) \sum_{m \in \Z}\widehat{W}(m, a_1)\diff x
		\ea
		By repeated integration by parts we see that the size of $ \widehat{W}(m, a_1)$ is negligible if $ m \gg \frac{lk X^{\ve}}{X} $, and therefore, we can truncate the sum over $m$ up to $\frac{lk X^{\ve}}{X}$ and add an error term  $O\l( X^{-100}\r)$. Thus, isolating the $m=0$ term and estimating the Fourier integral for the remaining terms trivially and then estimating the sum over $m, a_1, k$ and $l$ trivially, we obtain
		\ba
		M_{V}(X,r)&=\sum_{l|r}\sum_{k>0}\frac{\mu(k)}{k}\int V\l(\frac{x}{X}\r)\int V\l(\frac{lky}{X}\r)\sum_{a_1} \frac{1}{a_1}V\l(\frac{lka_1}{X}\r)\widehat{W}(0, a_1) \diff x \\
		&+ O\l(\tau (r)X^{1+\ve}\r). 
		\ea
		Again, applying Poisson summation, this time  to the $a_1$-sum, we obtain 
		\ba
		M_{V}(X,r) & =\sum_{l|r}\sum_{k>0}\frac{\mu(k)}{k}\iint V\l(\frac{x}{X}\r)V\l(\frac{lky}{X}\r)\sum_{h \in \Z}\widehat{U}(h)+ O\l(\tau (r)X^{1+\ve}\r),
		\ea 
		where $$U(z)=\frac{1}{z}V\l(\frac{lkz}{X}\r) \widehat{W}(0, z).$$
		As before, we isolate the contribution of the term $h =0$ and the total contribution of all the other terms is seen to be absorbed in the error term by the same kind of arguments as above. This finishes the proof. 
		
	\end{proof}
	\section{The Error Term: Initial steps}
	
	\subsection{Beginning the proof}
	Interchanging the order of summation, applying Poisson summation to the sum over $c_1$ after removing the coprimality condition (see Lemma \ref{6}) and separating the zero-th term on the dual side,  we obtain 
	\ba \label{BVsmooth}
	E_{V}(X,r) &= \sum_{l|r}\sum_{\substack{n\in \Z\\ n\neq 0}}\iint V\l(\frac{x}{X}\r)V\l(\frac{ly}{X}\r)\sum_{a_1} \frac{1}{a_1^2} V\l(\frac{la_1}{X}\r)V\l(\frac{r_1+yx}{a_1X}\r)\\
	&e\l(\frac{-nx}{a_1}\r)S(nr_1, 0, a_1)\diff x\, \diff y\\
	&+ \sum_{l|r}\sum_{\substack{m, n\in \Z\\ m, n\neq 0}}\iint  V\l(\frac{x}{X}\r)V\l(\frac{ly}{X}\r)\sum_{a_1} \frac{1}{a_1^2} V\l(\frac{la_1}{X}\r)V\l(\frac{r_1+yx}{a_1X}\r) \\
	&e\l(\frac{-nx-my}{a_1}\r)S(nr_1,-m, a_1)\diff x\, \diff y \\
	& = B_{0,V}(X,r) + B_{1,V}(X,r),
	\ea 
	say. Note that by repeated integration by parts, the $x$-integral, which appears in both the sums, is $O(X^{-100})$ unless $n \ll \frac{a_1 X^{\ve}}{X}\ll \frac{X^{\ve}}{l}$ for some $\ve>0$. Hence, we can restrict our attention to the sum over $n, l \ll X^{\ve}$, since if $l$ is bigger, the only term that survives has $n=0$ which has been accounted for already.
	Let us consider $B_{0,V} (X)$ first. 
	The Kloosterman sum appearing in $B_{0,V} (X)$ is the Ramanujan sum $r_{a_1}(nr_1)$. Using the standard  bound $|r_q(n)|\leq (q, n)$ (see \cite[Eq. (3.5)]{IK}) and estimating trivially using the support of $V$, 
	we  easily obtain
	\be \label{bVzero}
	B_{0,V} (X) \ll \tau(r) X^{1+\ve}. 
	\ee
	
	\subsection{Estimation of $B_{1,V}(X,r)$}
	Now our plan is to apply the Kuznetsov's formula (see Lemma \ref{kuz}) to the $a_1$-sum in $B_{1,V}(X,r)$ (see \eqref{BVsmooth}). Note that, as in the case of $n$, by repeated integration by parts, the $y$-integral is $O(X^{-100})$ unless $|m|\ll X^{\ve}$. Therefore, 
	\ba \label{B1smooth}
	B_{1,V}(X,r) &\ll \sum_{\substack{l|r\\l\ll X^{\ve}}}  \iint  V\l(\frac{x}{X}\r)V\l(\frac{ly}{X}\r)\frac{l}{X} 
	\Big| \sum_{\substack{|n|\ll \frac{X^{\ve}}{l}\\ n\neq 0}}\sum_{\substack{|m|\ll X^{\ve}\\ m\neq 0}} S_f(nr_1, -m)\Big|\diff x\, \diff y \\
	&+O(X^{-100}) ,
	\ea
	where
	\[
	S_f(nr_1, -m)= \sum_{a_1} \frac{1}{a_1} S(nr_1,-m, a_1)  f  \l(\frac{4\pi \sqrt{|mnr_1| }}{a_1}\r),
	\]
	with
	\begin{equation} \label{fv}
		f(t)= v\l(\frac{4\pi \sqrt{|mnr_1| }}{t}\r),
	\end{equation}  
	where
	\begin{equation} \label{vfunction}
		v(u) =v_{m, n, r_1, l,  X} (u) = \frac{X}{lu} V\l(\frac{lu}{X}\r)V\l(\frac{r_1+yx}{uX}\r) e\l(\frac{-nx}{u}\r)e\l(\frac{-my}{u}\r).
	\end{equation} 
	In the next section we obtain a strong estimate on the above sum over the variable $a_1$ which we have denoted by $c$ below.

	\section{Exploiting cancellations in the sum of Kloosterman sums}
	The goal of this section is to prove the following proposition and Theorem \ref{main2}. 
	\begin{proposition} \label{Kloosterman_sum} 
		For $r\ll X^2, l\ll X^{\ve}$ and $r_1=r/l$, we have,
		\ba \label{character_equation}
		\sum_{\substack{|n|\ll \frac{X^{\ve}}{l}\\ n\neq 0}}\sum_{\substack{|m|\ll X^{\ve}\\ m\neq 0}}  \sum_{c>0} \frac{1}{c} S(nr_1,-m, c)  f_{{}}  \l(\frac{4\pi \sqrt{|mnr_1| }}{c}\r) \ll  r^{\theta } X^{\ve},
		\ea
		where $f$ is as given in  \eqref{fv} and the above bound holds uniformly over the variables $x, y \in [X, 2X]$. 
	\end{proposition}
	\subsection{Decomposition of the spectral sum}
	We apply the Kuznetsov formula (Lemma \ref{kuz}) to the $c$-sum and obtain
	\ba \label{split}
	\sum_{c} \frac{1}{c} S(nr_1,-m, c)  f_{{}}  \l(\frac{4\pi \sqrt{|mnr_1| }}{c}\r) = \Sigma _{\text{Maass}}  + \Sigma _{\text{cont.}}+\Sigma_{\text{hol.}}, 
	\ea 
	where 
	\ba \label{discrete}
	\Sigma _{\text{Maass}}  := \sum_{j = 1}^{\infty}\rho_j(nr_1)\overline{\rho_j(m)}\check{f}(\kappa _ j )
	\ea 
	if $mnr_1>0$ and $\check{f}(\kappa _ j )$ is replaced by $\ddot{f}(\kappa _ j )$ if $mnr_1<0$;
	\ba \label{cont}
	\Sigma _{\text{cont.}} :=  \frac{1}{\pi}\int\limits_{-\infty}^{\infty} (nmr_1)^{-i\eta}\sigma_{2i\eta}(nr_1)\overline{\sigma_{2i\eta}(m)}\frac{\cosh(\pi \eta)\check{f}(\eta)}{\l|\zeta (1+2i\eta)\r|^2}\diff \eta,
	\ea 
	if $mnr_1>0 $ and $\check{f}(\eta)$ is replaced by $\ddot{f}(\eta)$ if $mnr_1<0$; 
	and
	\ba\label{hol}
	\Sigma_{\text{hol.}} :=\sum_{0<k\equiv 0(\textnormal{mod }2)} \widetilde{f}(k)\sum_{1\leq j \leq \textnormal{dim} S_k(\Gamma)} \overline{\psi_{jk}}(|m|)\psi_{jk}(|nr_1|),
	\ea
	and this last sum is non-existent if $mnr_1<0$.

	\subsection{Analysis of $\check{f}_ {{}} (\eta)$}
	\begin{proposition} \label{fcheckprop}
		Suppose $\eta\neq 0$ is a real number. For all sufficiently small $\ve>0$, we have the bound 
		\ba 
		\check{f} (\eta) \ll X^{\ve} \ \frac{e ^{ -\pi|\eta|}}{ |\eta| ^{5/2}},
		\ea 
		where $f$ is defined in \eqref{fv} and $m$ and $n$ satisfy the conditions appearing in Prop \ref{Kloosterman_sum}. 
	\end{proposition}
	\begin{proof}
		By the Mellin inversion formula applied to the well-known formula for the Mellin transform of the $K$-Bessel function (see \cite[App B.4]{IW2}), we can write
		\ba \label{Kbessel}
		K_{2i \eta }(t) = \frac{1}{2\pi i} \int\limits_{(\sigma)} 2^{s-1}t ^{-s} \Gamma\l(\frac{s}{2}+i\eta  \r)\Gamma\l(\frac{s}{2}-i\eta  \r) \diff s, \quad  \text{ for } \Re s =\sigma >0.
		\ea
		Hence, 
		\ba \label{fcheckint}
		\check{f}_ {{}} (\eta) = \frac{4}{\pi}\int\limits_ 0 ^{\infty} \frac{1}{2\pi i} \int\limits_{(\sigma)} 2^{s-1}t ^{-s} \Gamma\l(\frac{s}{2}+i\eta  \r)\Gamma\l(\frac{s}{2}-i\eta  \r) f_ {{}}(t) \diff s \frac{\diff t}{t}.
		\ea 
		Note that here the variable $t$ ranges over a compact set away from zero due to the support of $v$ (see \eqref{fv}). Also, by the Stirling asymptotics the Gamma factors exponential decay in the imaginary part of the variable $s$. Hence the integral is absolutely convergent for any $\sigma\in (-2, 0)$.
		Now we shift the line of integration to $\sigma = -1-\delta$, where $0<\delta<1/10$ is a positive real number which we shall take as small as we want, and in this process we pick up residues from the poles at $s = \pm 2i \eta $ and by \eqref{fcheckint}. Thus we have,
		\ba 
		\check{f}_ {{}} (\eta) &= \frac{4}{\pi} \frac{1}{2\pi i}  \int\limits_ 0 ^{\infty}  2^{2i\eta-1}t ^{-2i\eta} \Gamma\l( 2i\eta  \r) f_ {{}}(t)  \frac{\diff t}{t} \\
		&+ \frac{4}{\pi} \frac{1}{2\pi i} \int\limits_ 0 ^{\infty}  2^{-2i\eta -1}t ^{2i\eta} \Gamma\l(-2 i\eta  \r) f_ {{}}(t)  \frac{\diff t}{t}\\
		&+\frac{4}{\pi}\int\limits_ 0 ^{\infty} \frac{1}{2\pi i} \int\limits_{(-1-\delta)} 2^{s-1}t ^{-s} \Gamma\l(\frac{s}{2}+i\eta  \r)\Gamma\l(\frac{s}{2}-i\eta  \r) f_ {{}}(t) \diff s \frac{\diff t}{t}.
		\ea 
		Moreover, applying integration by parts twice in the integral over $t$, we obtain 
		\ba \label{firstrepeated}
		\int\limits_0^{\infty} t^{-s-1}  f _{{}} \l( t \r) \diff t = \frac{1}{s(s-1)}  \int\limits_0^{\infty} t^{-s+1}  f ''_{{}} \l( t \r) \diff t,
		\ea
		and thus  
		\ba \label{fcheckintone}
		\check{f}_ {{}} (\eta) & =  \frac{4}{\pi} \frac{1}{2\pi i} \int\limits_ 0 ^{\infty}   \frac{1}{2i\eta (2i\eta -1)} 2^{- 1 + 2i \eta }t ^{-2i \eta +1} \Gamma\l( 2i\eta  \r)f''_ {{}}(t)  \diff t \\
		& -  \frac{4}{\pi} \frac{1}{2\pi i} \int\limits_ 0 ^{\infty}   \frac{1}{2i\eta (2i\eta + 1)} 2^{-1 - 2i \eta }t ^{2i \eta +1} \Gamma\l( -2i\eta  \r)f''_ {{}}(t)  \diff t \\
		& + \frac{4}{\pi}  \frac{1}{2\pi i}   \int\limits_ 0 ^{\infty}  \int\limits_{(-1-\delta )} \frac{1}{s(s-1)} 2^{s-1}t ^{-s+1} \Gamma\l(\frac{s}{2}+i\eta  \r)\Gamma\l(\frac{s}{2}-i\eta  \r) f''_ {{}}(t) \diff s \diff t \\
		& = \mathcal{A}_1 + \mathcal{A}_2 +\mathcal{A}_3, 
		\ea 
		say. To estimate the terms, we first note that by \eqref{vfunction},
		\ba 
		v '(u) \ll  \frac{u^{\ve}}{u} ,\quad v ''(u) \ll  \frac{u^{\ve}}{u^2},
		\ea
		and  consequently by \eqref{fv},
		\be \label{firstderivative}
		f'(t) = -v'\l(\frac{4\pi \sqrt{|mnr_1| }}{t}\r)\frac{4\pi \sqrt{|mnr_1| }}{t^2} \ll 
		\frac{X^{\ve}}{t}. 
		\ee
		and similarly,
		\ba \label{secondderivative}
		f''(t)  \ll \frac{X^{\ve}}{t^2}.
		\ea 
		Hence, for the first two terms  $\mathcal{A}_1$ and $\mathcal{A}_2$ in \eqref{fcheckintone}, we have, by the Stirling asymptotic (see Lemma \ref{stir}),
		\ba \label{A1}
		\mathcal{A}_1, \mathcal{A}_2 \ll \frac{e ^{ -\pi|\eta|}X^{\ve}}{ |\eta| ^{5/2}}.
		\ea 
		Now we will focus on $\mathcal{A}_3$. For that we first put $s = -1-\delta + iw$ and we have 
		\ba 
		\mathcal{A}_3 &= \frac{4}{\pi}  \frac{1}{2\pi i}   \int\limits_ 0 ^{\infty}  \int\limits_{ - \infty} ^ {\infty}\frac{2^{-2-iw}t ^{2+\delta-iw}}{(-1-\delta +iw)(-2-\delta+iw)}  \Gamma\l(\frac{-1-\delta+iw}{2}+i\eta  \r) \\
		&\Gamma\l(\frac{-1-\delta+iw}{2}-i\eta  \r) f''_ {{}}(t) \diff w \diff t. \\
		\ea 
		To estimate the $w$-integral, we use the Stirling asymptotic for the two Gamma factors (except in two small 
		intervals around  $2\eta$ and $-2\eta$ where we use it for one Gamma factor and use the trivial bound for the other). Thus we have, 
		\ba 
		\mathcal{A}_{3} &\ll     \frac{e ^{ -\pi|\eta|}}{ |\eta| ^{2+\delta/2}}\int\limits_ 0 ^{\infty}  t ^{2+\delta}\left|f''(t)\right|
		\diff t .
		\ea
		From \eqref{fv} and \eqref{vfunction}, it follows that unless
		\ba 
		t \asymp \frac{l \sqrt{|mnr_1| }}{X},
		\ea
		the function $V$ takes the value zero.
		Therefore, by \eqref{secondderivative},
		\ba \label{A3}
		\mathcal{A}_{3} 
		& \ll \l(\frac{l\sqrt{|mnr_1| }}{X}\r)^{1+\delta} \frac{e ^{ -\pi|\eta|} X^{\ve}}{ |\eta| ^{2+\delta/2}} \\
		& \ll X^{\ve(2+\delta)}  \frac{e ^{ -\pi|\eta|}}{ |\eta| ^{2+\delta/2}},
		\ea
		since $r\ll X^2$.  The proposition follows by combining \eqref{fcheckintone}, \eqref{A1} and \eqref{A3}
		by choosing $\delta$ and then redefining $\ve$ suitably. 
	\end{proof}
	\subsection{Analysis of $\ddot{f}(\eta)$}

	\begin{proposition} \label{fhatprop}
		For all sufficiently small  $\ve>0$, we have 
			\ba \label{fhatzero}
		\ddot{f} (\eta) \ll X^{\ve}{(\cosh \pi\eta)}^{-1} 
		\ea 
		for all $\eta$; and 
		\be\label{fhat}
		\ddot{f} (\eta) \ll X^{\ve} \ \frac{e ^{ -\pi|\eta|}}{ |\eta| ^{5/2}}
		\ee 
			for $|\eta|>1$.
	\end{proposition}	
	\begin{proof}
	We first note that the support of the function $V$ restricts the size of the argument in the function $f$ (see \eqref{fv} and \eqref{vfunction}). Precisely, we get
	\[
	\frac{2\pi l \sqrt{|mnr_1|}}{X}\leq \frac{4\pi \sqrt{|mnr_1| }}{c}\leq
	 \frac{4\pi l \sqrt{|mnr_1|}}{X}.
	\]
		Since $r_1=r/l$ and $r\ll X^2$, we see that 
		\[
		T:=\frac{4\pi l\sqrt{|mnr_1|}}{X} \ll \sqrt{mnl}\ll X^{\ve}. 
		\]
		Hence, the first bound in the proposition follows directly from (7.2)  of \cite[Lemma 7.1]{DI1}.   On the other hand, for $|\eta|\geq 1$,  (7.4) of \cite[Lemma 7.1]{DI1} implies
		\ba 
		\ddot{f} (\eta) \ll X^{\ve}{(\cosh \pi\eta)}^{-1} \l(|\eta| ^{-3/2} + |\eta|^{-2} T\r){|\eta|}^{-1},
		\ea 
from which the second bound follows. Note that the integral transforms are defined a little differently in \cite{DI1} in comparison to \cite{IK}, which we have followed. 
%Also, note that in order to satisfy the bounds (1.44) of \cite{DI1}, we may divide  the function $f$ by $X^{\ve}$. 

	\end{proof}
	
	\subsection{Contribution of  the Maass forms part of the spectrum} 
	First, we prove a preparatory lemma. 
	\begin{lemma}\label{Rankin-Selberg}
		For any integer $q\geq 1$, the following bound holds. 
		\be\label{RS}
		\sum_{ 1 \leq n \leq N }\frac{|\rho_j(nq)|^2}{\cosh(\pi \kappa_j)}  \ll q^{\theta+\ve} N(N\kappa_j)^{\ve}.
		\ee
		Here $\ve>0$ is arbitrary and the implied constant depends on $\ve$. 
	\end{lemma}
	\begin{proof}
		by Lemma \ref{hl}, \eqref{realtion} and the Hecke relation \eqref{heckerelation}, we have
		\ba \label{N}
		\sum_{1 \leq n \leq N }\frac{|\rho_j(nq)|^2}{\cosh(\pi \kappa_j)} 
		&\ll \kappa_j^{\ve} \sum_{ 1 \leq n \leq N }\l|\sum_{k|(n,q)}\mu(k)\lambda _j\l(\frac{q}{k}\r)\lambda _j\l(\frac{n}{k}\r) \r| \\
		&\ll \kappa_j^{\ve}  q^{\theta+\ve} \sum_{k|q} \sum_{ 1 \leq n \leq \frac{N}{k} }\l|\lambda _j\l(n\r)\r|^2  \\
		&\ll q^{\theta+\ve} N(N\kappa_j)^{\ve},
		\ea
		by Remark  \ref{kim} and Lemma \ref{hec}. 
	\end{proof}
	Let us introduce the notation  $$\mathcal{M}(X)= \sum_{\substack{n \ll \frac{X^{\ve}}{l}\\ n\neq 0}}\sum_{\substack{m \ll X^{\ve}\\ m\neq 0}} \Sigma_{Maass},$$
	$\Sigma_{Maass}$ is defined in \eqref{split}. 
	\begin{proposition} \label{Drprop}
		We have, for any $\ve>0$
		\ba 
		\mathcal{M}(X) \ll r^{\theta} X^{\ve},
		\ea
		where the implied constant depends only on $\ve>0$. 
	\end{proposition}
	\begin{proof}
		By \eqref{discrete} and Propositions \ref{fcheckprop} and \ref{fhatprop}, 
		\ba 
		\mathcal{M}(X) \ll  X^{2\ve}  \sum_{\substack{n \ll \frac{X^{\ve}}{l}\\ n\neq 0}}\sum_{\substack{m \ll X^{\ve}\\ m\neq 0}}   \sum_{j = 1}^{\infty}|\rho_j(nr_1)\rho_j(m) | \frac{e ^{ -\pi\kappa_j}}{ \kappa _j ^{5/2}},
		\ea

		We subdivide the $j$-sum into dyadic pieces with $\kappa_j \sim K$. Since
		 $|\kappa_j|>1$, we may assume that $K\geq 1$. Applying  the Cauchy-Schwarz inequality to the $m$-sum and the $n$-sum, we obtain
		\ba \nonumber
		\mathcal{M}(X)  &\ll X^{2\ve}\sum_{ K \textnormal{ dyadic}}\frac{1}{K^{5/2}} \sum_{ \kappa _j \sim K} \frac{ \cosh(\pi \kappa_j)}{e^{\pi \kappa_j}}\l(   \sum_{\substack{n \ll \frac{X^{\ve}}{l}\\ n\neq 0}}  \frac{|\rho_j(nr_1)|^2}{\cosh(\pi \kappa_j)}\r)^{\frac{1}{2}}
		\l(\sum_{\substack{ m \ll X^{\ve}\\ m\neq 0}} \frac{|\rho_j(m)|^2}{\cosh(\pi \kappa_j)}\r)^{\frac{1}{2}}\\
		&\ll  r_1^{\theta+\ve_1}  X^{2\ve}\sum_{K \textnormal{ dyadic}}\frac{1}{K^{5/2}} \sum_{ \kappa_j \sim K} (\kappa_j)^{\ve_1},
		\ea  \nonumber
		by applying Lemma \ref{Rankin-Selberg} to sums over $m$ and $n$. Here $\ve_1>0$ is arbitrary.
		Finally, an application of the Weyl law (see Lemma \ref{weyl}) applied to the $\kappa_j$-sum yields the bound
		\ba \label{kappa}
		\mathcal{M}(X) &\ll r_1^{\theta+\ve_1}  X^{2\ve}\sum_{K \textnormal{ dyadic} } K^{-1/2-\ve+\ve_1}\\
		\ea
		By choosing $\ve_1$ sufficiently small, we can ensure that the sum over $K=2^j$ converges. Redefining $\ve$, we obtain the proposition. 
	\end{proof}
	\subsection{Contribution of the continuous spectrum}
	\begin{proposition} \label{Crprop}
		We have, for any sufficiently small $\ve>0$, 
		\ba 
		\Sigma _{\text{cont.}}\ll  X^{\ve},
		\ea
		where the implied constant depends only on $\ve$. 
	\end{proposition}
	\begin{proof}
 By \eqref{cont} and symmetry, it is enough to estimate the integral
		\[
		\int\limits_{0}^{\infty} \frac{\cosh(\pi \eta)\check{f}(\eta)}{\l|\zeta (1+2i\eta)\r|^2}\diff \eta,
		\]
		as the contribution of the other terms can be estimated trivially and all together these produce a factor of size $O\l(X^{\ve}\r)$. We need to use the estimates on $\check{f}$ and $\ddot{f}$ according as $mnr_1$ is positive or negative and we apply Proposition \ref{fcheckprop} and Proposition \ref{fhatprop}  to estimate them. In the second case,
		we divide the integral into two parts: \(\int_0^1+\int_1^{\infty}\). We apply the first bound \eqref{fhatzero} of Proposition \ref{fhatprop} for estimating the first integral. The contribution from second integral can be absorbed in contribution from the integral corresponding $\ddot{f}$ that appears in the first case $mnr_1>0$, since the bounds for $\check{f}$ and $\ddot{f}$ are identical in the range $|\eta|>1$. Thus, we obtain
		\ba 
		\Sigma _{\text{cont.}}  \ll  X^{\ve}\Biggl(  \int\limits _{0}^{1} \frac{1}{\l|\zeta (1 + 2i\eta)\r|^2}+\int\limits _{0}^{ \infty}  \frac{e^{-\pi \eta }\cosh(\pi \eta)}{\eta^{5/2}\l|\zeta (1 + 2i\eta)\r|^2}\diff \eta\Biggr).
		\ea
	The first integral is obviously convergent and  contributes $O(1)$. As for the second integral, 
		we break it into two parts and rewrite it as
		\ba
		 \int_0^{\ve}\frac{e^{-\pi \eta }\cosh(\pi \eta)}{\eta^{5/2}\l|\zeta (1 + 2i\eta)\r|^2}\diff \eta
		+\int_{\ve}^ {\infty} \frac{e^{-\pi \eta }\cosh(\pi \eta)}{\eta^{5/2}\l|\zeta (1 + 2i\eta)\r|^2}\diff \eta.
		\ea
		For $0 < \eta <\ve $,
		\be
		\eta^{5/2}\zeta (1 + 2i\eta)^2 = \eta^{5/2}\l( \frac{1}{ 2i\eta}+\gamma+ O\l( \eta \r)\r)^2\gg \eta^{1/2},
		\ee
		by  \cite[(2.1.16)]{T}. 
		Also, from \cite[(3.6.5)]{T}, for $\eta \geq \ve$,
		\be \frac{1}{\l|\zeta (1 + 2i\eta)\r|} \ll |\log \eta| ^7. \ee
		Therefore
		\ba
		\Sigma _{\text{cont.}} & \ll    X^{\ve}\left(1+\int_0^{\ve} \frac{1}{\eta ^{1/2}} \diff \eta 
		+ \int_{\ve}^ {\infty} \frac{|\log \eta| ^{14}}{\eta^{5/2}}\diff \eta\right) \\
		&\ll X^{\ve}.
		\ea
	The proposition follows.
	\end{proof}
	
	\subsection{Contribution of the holomorphic forms part of the spectrum}
	At first, we state a lemma concerning sums of $J$-Bessel functions that will be invoked later.
	\begin{lemma} 
		We have, 
		\be\label{B1}
		(a)  \sum_{0<k \equiv 0 (\textnormal{mod } 2)} 2(k-1)J_{k-1}(x)J_{k-1}(y) = xy\int_0^1 uJ_0(ux)J_0(uy) \diff u,
		\ee
		and 
		\be\label{B2}
		(b) \sum_{0<k\equiv 0(\textnormal{mod }2)} (k-1)i^{-k}J_{k-1}(x) =-\frac{x}{2}J_0(x).
		\ee
	\end{lemma}
	\begin{proof}
		The first one is eqn.~(B.51) in \cite[App.~B4]{IW2}. The second one follows from the relation (op. cit.)
		\[
		z(J_{\nu-1}(z) +J_{\nu+1}(z))  =2 \nu J_{\nu}(z) 
		\]
		applied to each term of the left-hand side and this transforms the original series to a telescopic series summing up to the right hand side. 
	\end{proof} 
	
	\begin{proposition}\label{holprop}
		We have,for any $\ve>0$,
		\be
		\Sigma_{\text{hol.}} \ll X^{\ve},
		\ee
		where the implied constant depends only on $\ve$. 
	\end{proposition}
	\begin{proof} 
		Following the notations introduced in \S \ref{holo}, we have 
		\begin{align*}
			&\Sigma_{\text{hol.}}=\sum_{0<k\equiv0(\textnormal{mod }2)} \widetilde{f}(k)\frac{(4\pi)^{k-1}}{\Gamma(k-1)}\sum_{1\leq j \leq \textnormal{dim} S_k(\Gamma)}\overline{\rho_{h_{j,k}}(|m|)} \rho_{h_{j,k}}(|nr_1|).
			\end{align*}
			Now, by the Petersson formula (see \ref{Petersson}) applied to the last sum,  we see that the above sum is 
			\begin{align*}
			 \int_0^{\infty}&\frac{f(y)}{y}    \sum_{0<k\equiv 0(\textnormal{mod }2)}\frac{(k-1)}{\pi}i^{-k} J_{k-1}(y)\\
			 &\times \left\{\delta(|m|, |nr_1|)+2\pi i^k \sum_{c>0}\frac{S(|m|,|nr_1|;c)}{c}J_{k-1}\left(\frac{4\pi\sqrt{|mnr_1|}}{c}\right)\right\}\diff y.
			\end{align*}
			Now, by \eqref{B1} and\eqref{B2},
			the above sum is
		\begin{align*}
		\ll\int_0^{\infty} f(y)\Biggl(&\delta(|m|, |nr_1|)|J_0(y)|\\
		&+\sum_{c>0}\frac{|S(|m|,|nr_1|;c)|}{c} \frac{4\pi\sqrt{|mnr_1|}}{c}\int_0^1 \left|uJ_0\left(u\frac{4\pi\sqrt{|mnr_1|}}{c}\right)J_0(uy)\right|\diff u\Biggr) \diff y
		\end{align*}
		
		Now the standard bound $|J_{\nu}(x)|\leq 1$ for real $x$ and the Weil bound for Kloosterman sums shows that the above is $$\ll X^{\ve}\int_0^{\infty} f(y) \diff y \ll X^{\ve}.$$
	\end{proof}
	
	\subsection{Proof of Proposition \ref{Kloosterman_sum}}  
	Bringing together \eqref{split}, Prop. \ref{Drprop}, Prop. \ref{Crprop}, and Prop. \ref{holprop},  Prop. \ref{Kloosterman_sum} follows.
	
	\section{The final steps} 
	By Prop. \ref{Kloosterman_sum}, the bound 
	\be \label{bVone}
	B_{1,V} (X) \ll r^{\theta} X^{1+\ve}. 
	\ee
	follows immediately. 
	Thus we obtain the following bound on $E_{V}(X,r)$.
	\begin{proposition}  \label{BVX}
		We have,
		\begin{equation*}
			E_{V}(X,r) \ll r^{\theta}X^{1+\ve}.
		\end{equation*}
	\end{proposition}
	This proposition follows from  \eqref{bVzero} and \eqref{bVone}.
	\subsection{Proof of Theorem \ref{main2}} The theorem follows from Prop. \ref{MTsmooth} and Prop. \ref{BVX} after absorbing the term $ O\l(\tau (r)X^{1+\ve}\r)$ into $O\l(r^{\theta}X^{1+\ve}\r)$. This can be done since $\tau(r)=O\l(X^{\ve}\r)$.

	\subsection{Proof of Corollary \ref{maincorollary}}
	Let us write  $\alpha=r/ X^2$. Note that $0<|\alpha|<3$. Then, substituting $ x=Xu, y=(lk)^{-1}Xv, z=(lk)^{-1}Xt$, we have
	\begin{align*}
		M_V(X,r) = \sum_{l|r}\sum_{k>0}\frac{\mu(k)}{lk^2}\iiint \frac{1}{t} V\l(u\r)V\l(v\r)V\l(t\r)V\l(\frac{\alpha}{t} + \frac{uv}{t}\r) \diff u \diff v \diff t.
	\end{align*}
	Now, by the Mean Value Theorem of Calculus,
	\begin{align*}
		V\l(\frac{\alpha}{t} + \frac{uv}{t}\r) = V\l(\frac {uv}{t}\r)+O \l(\frac{|\alpha|}{t}||V'||_{\infty}\r),
	\end{align*}
	and the second  term is $O\l(|r|/X^2 \r) $. This leads immediately to the corollary. 
	
	%%%%%%%%%%%%%%%%%%%%%%%%%%%%%%%%%%%%%%%%%%%%%%%
	\section{ Proof of Theorem \ref{modpthm}}
	We begin with a simple lemma.
	\begin{lemma}\label{sum-W-hat} (a) Let $W$ be at least twice continuously differentiable function with compact support inside $[1,2]$. Then for any $x>0$, we have
		\be\label{FTsum}
		\sum_{\substack{n \in \Z\\n\neq 0}}\left|\widehat{W}\l(\frac n x\r)\right|=O\l( x\r).
		\ee
	(b) If $W$ is a function with compact support inside $[1,2]$ which is at least thrice continuously differentiable, then for any $x>0$, we have
		\be\label{FTsum2}
		\sum_{\substack{n \in \Z\\n\neq 0}}|n|\left|\widehat{W}\l(\frac n x\r)\right|=O\l( x^2\r).
		\ee
	\end{lemma}
	\begin{proof}
We divide the sum into two parts: $|n|\leq x$ and $|n|> x$. The first part is obviously $O(x)$.
		To estimate the second part we utilise the bound $\widehat{W}(y)\ll |y|^{-2}$ obtained by integration by parts twice. Thus the second sum is
		\begin{align*}
			&\ll \sum_{|n|>x}{\l|n/x\r|}^{-2}\\
			&\ll x. 
		\end{align*}
		Note that if $0<x<1$, the first sum is non-existent and the second sum is actually $O\l(x^2\r)$. This finishes the proof of part (a).
		The proof of part (b) is similar and the only difference is that we apply integration by parts thrice in the second part, obtaining $\widehat{W}(y)\ll |y|^{-3}$. 
	\end{proof}

	Now we proceed with the proof of Theorem \ref{modpthm}. 
	\begin{proof}
	Let us introduce the notation 
	\[
	V_X(a, b, c, d)= V\l(\frac{a}{X}\r)V\l(\frac{b}{X}\r)V\l(\frac{c}{X}\r)V\l(\frac{d}{X}\r).
	\]
By orthogonality of additive characters modulo $p$,
	\ba
& \mathop{\sum\sum\sum\sum}_{\substack{  ad-  bc\equiv 1 \modp}}V_X(a, b, c, d)
	\\
	&=\frac{1}{p} \mathop{\sum_a\sum_b\sum_c\sum_d}V_X(a, b, c, d)\\
	&+\frac{1}{p}\starsump e\l(\frac{-h}{p}\r) \mathop{\sum_a\sum_b\sum_c\sum_d}V_X(a, b, c, d) e\l(\frac{h(  ad-  bc)}{p}\r)\\
	&=M +E,
	\ea
	say, where the notation $\starsum$ indicates that the sum is restricted to the nonzero residue classes modulo $p$. 
	By Poisson summation and using the bound 
	\be\label{decay}
	\widehat{V}(y)\ll_j |y|^{-j}, \ y\neq 0;
	\ee
for any integer $j\geq 1$  which is obtained by  integration parts $j$ times in the integral defining $\widehat{V}(y)$, we have
	\be\label{mt}
	M=\frac{X^4}{p}{\widehat{V}(0)}^4+O(X^{-100}).
	\ee
	As for the sum $E$, the natural step now would be to apply Poisson summation to all the four variables. However, doing so leads to a similar sum as before and does not seem to help in estimating it efficiently. Instead, we apply it only to two variables appearing in the two different monomials, say $c$ and $d$. Thus
	\ba
	E=\frac{X^2}{p}\starsump e\l(\frac{-h}{p}\r)\mathop{\sum_a\sum_b\sum_c\sum_d}_{d\equiv -h a \modp, \ c\equiv hb \modp}V\l(\frac{a}{X}\r)V\l(\frac{b}{X}\r) 
	\widehat{V}\l(\frac{cX}{p}\r)\widehat{V}\l(\frac{dX}{p}\r),
	\ea
	where we have denoted the dual variables again by $c$ and $d$, respectively. Note that in the dual side, there is no term corresponding to $c=0$ or $d=0$. This is because $a, b\not\equiv 0 \modp $ due to the support of $V$ and our choice of $X$. Summing over  $h$ trivially, we can write
	\ba 
E \leq \frac{X^{2}}{p} \mathop{\sum_{a}
	\sum_{b}
	\sum_{c \neq0 }
	\sum_{d \neq 0}}_{ac+bd \equiv 0\bmod p} V \left(\frac{a}{X}\right)V \left(\frac{b}{X}\right)
\l| \widehat{V}\left(\frac{c X}{p}\right)
\widehat{V}\left(\frac{d X}{p}\right)\r|,
\ea 
We rewrite the above sum as
\ba
S := \mathop{\sum_{a}
	\sum_{b}
	\sum_{c \neq0 }
	\sum_{d \neq 0}\sum_{m \in \Z}}_{ac+bd = mp} V \left(\frac{a}{X}\right)V \left(\frac{b}{X}\right)
\l|\widehat{V}\!\left(\frac{c X}{p}\right)
\widehat{V}\!\left(\frac{d X}{p}\right)\r|,
\ea
Let \( k=(b,c) \), \( b = k b_1 \) and \( c = k c_1 \). Eliminating the variable $d$ by  interpreting the equality 
\(ac+bd = mp\) as a congruence modulo $b_1$ and dropping the condition that $d\neq 0$ by positivity, we obtain 
\[
S
\leq \sum_{m \in \Z}\sum_{k \mid mp}
\mathop{\sum_{b_1}\sum_{c_1 \neq 0}}_{\substack{(b_1,c_1)=1}}
V\left(\frac{k b_1}{X}\right) \l| \widehat{V}\left(\frac{kc_1 X}{p}\right)\r|
\sum_{\substack{a \equiv -mp \overline{c_1}/k \bmod{b_1}}}
V\left(\frac{a}{X}\right)\l|
\widehat{V} \left(\frac{m p - ak c_1}{k b_1}\cdot \frac{X}{p}\right)\r|.
\]
Recall that \( X < b < 2X <p \) by our assumptions. Hence  \( (k,p)=1 \). Therefore, the condition  \( k \mid mp\) forces that \( k \mid m \). Let us write \( m = k\ell \), where \(\ell \in \Z\). Thus we obtain
\[
S \leq \sum_{1 \le k \le 2X}
\sum_{b_1}
V\left(\frac{k b_1}{X}\right)
\sum_{c_1 \neq 0 }
\left| \widehat{V} \left(\frac{k c_1 X}{p}\right) \right|
\sum_{\ell \in \mathbb{Z}}
\sum_{a \equiv -\ell p \ov{c_1} (\bmod b_1)}
\left| V\left(\frac{a}{X}\right)
\widehat{V}\left(\frac{\ell X}{b_1} - \frac{a c_1 X}{b_1 p}\right) \right|.
\]
We now  subdivide the above sum according as $\ell c_1<0$ or $\ell c_1\geq 0$, and consider the two cases separately. We denote these two sums by $S_{\ell c_1<0}$ and $S_{\ell c_1\geq 0}$, respectively.\\ 
Case (i): \( \ell c_1 < 0 \). 
Here $\ell$ and $c_1$ are of different sign and therefore, by the bound \eqref{decay}, we have 
\[
\widehat{V}\!\left(\frac{\ell x}{b_1} - \frac{a c_1 X}{b_1 p}\right)
\ll_j \left(\frac{|\ell| x}{b_1} + \frac{a |c_1| X}{b_1 p}\right)^{-j}
\ll \left(\frac{b_1}{|\ell|X}\right)^j
\ll \left(\frac{1}{k |\ell|}\right)^j.
\]
By summing over \( a \) trivially and dropping the condition $\ell c_1<0$ by positivity, we have, 
\begin{align*}
S_{lc_1<0}& \ll
X
\sum_{1 \le k \le 2X}
\sum_{b_1 }
\frac{1}{b_1}V \left(\frac{k b_1}{X}\right)
\sum_{c_1 \neq 0 } \left| \widehat{V} \left(\frac{k c_1 X}{p}\right) \right|
\sum_{\ell \neq 0 }
\left(\frac{1}{k|\ell|}\right)^j\\
&\ll p, 
\end{align*}
by \eqref{FTsum} applied to the sum over $c_1$ and summing over $b_1$ and $k$ trivially. \\

\noi
Case (ii): \( \ell c_1 \geq 0 \) . 
Here $\ell$ and $c_1$ must have the same sign. Without loss of generality, we assume \( \ell , c_1\geq 0 \). We decompose  the inner sum over $\ell$ into two parts: \(\ell \leq L\) and \(\ell >L\), where 
\[ 
L := \frac{4X c_1}{p}.
\]
Suppose first that $L\geq 1$. For the part where $\ell\leq L$, we apply the trivial bound on $\widehat{V}$ and then sum over $a$ and $\ell$ trivially and this gives
\[
 \sum_{0 \leq \ell \leq L}  \sum_{a \equiv -\ell p \ov{c_1} (\bmod b_1)}
\left| V\left(\frac{a}{X}\right)
\widehat{V}\left(\frac{\ell X}{b_1} - \frac{a c_1 X}{b_1 p}\right) \right| \ll \frac{X}{b_1} L\ll \frac{X^2 c_1}{p b_1}.
\]
Now summing over $c_1$ using \eqref{FTsum2} and summing over $b_1$ and $k$ trivially, we  see that this part of $S_{\ell c_1\geq 0}$ contributes $O(p)$.\\
Now, we consider the case $0<L<1$. In this case $\ell \leq L$ means that only the term $\ell=0$ survives and again a trivial estimation yields the bound $O(p)$ for the contribution of this term.\\
\noindent
For $\ell > L $, our choice of $L$ ensures that
\ba 
\left|\frac{\ell X}{b_1} - \frac{a c_1 X}{b_1 p}\right| \geq  \frac{\ell X}{2b_1}. 
\ea 
Hence,  
\ba 
\widehat{V}\!\left(\frac{\ell x}{b_1} - \frac{a c_1 X}{b_1 p}\right)
\ll_j \left(\frac{b_1}{\ell x}\right)^j
\ll \left(\frac{1}{k \ell}\right)^j.
\ea 
Using the same analysis as in case (i), we see that this part also contributes $O(p)$. 
If both $\ell$ and $c_1$ are  negative then the same proof applies. Thus we have proved
\[
S_{\ell c_1\geq 0}\ll p,
\]
and hence $S\ll p$, which yields $E\ll X^2$. This finishes the proof.
\end{proof}.

	\bibliographystyle{alphaurl}
	\bibliography{reference}
	%%%%%%%%%%%%%%%%%%%%%%%%%%%%%%%
\end{document}